\documentclass[12pt, twoside]{amsart}
\usepackage{amssymb}
\usepackage{amsthm}
\usepackage{amsmath}
\usepackage{amsxtra}
\usepackage{latexsym}
\usepackage{mathrsfs}
\usepackage[all,cmtip]{xy}
\usepackage[all]{xy}
\usepackage{enumitem}
\usepackage{xcolor}
\usepackage{comment}
\usepackage{mathabx,epsfig}
\def\acts{\ \rotatebox[origin=c]{-90}{$\circlearrowright$}\ }
\def\racts{\ \rotatebox[origin=c]{90}{$\circlearrowleft$}\ }

\newtheorem{thm}{Theorem}[section]
\newtheorem{lem}[thm]{Lemma}
\newtheorem{conj}[thm]{Conjecture}
\newtheorem{claim}[thm]{Claim}
\newtheorem{prop}[thm]{Proposition}
\newtheorem{cor}[thm]{Corollary}
\newtheorem{eg}[thm]{Example}
\newtheorem{question}[thm]{Question}
\theoremstyle{definition}
\newtheorem{defn}[thm]{Definition}
\newtheorem{rmk}[thm]{Remark}
\newtheorem*{ack}{Acknowledgement}
\numberwithin{equation}{section}

\def\Q{{\mathbb Q}}

\def\P{{\mathbb P}}

\def\NE{\overline{NE}}

\DeclareMathOperator{\id}{id}

\DeclareMathOperator{\ord}{ord}

\newenvironment{claimproof}[0]
  {%
   \paragraph{\it Proof.}%
  }
  {%
    \hfill$\blacksquare$%
  }

  {\end{list}}

\title[]
{Structure of Fano fibrations of varieties admitting an int-amplified endomorphism}
\author{Shou Yoshikawa}
\address{Graduate school of Mathematical Sciences, the University of Tokyo, Komaba, Tokyo,
153-8914, Japan}
\email{yoshikaw@ms.u-tokyo.ac.jp}

\begin{document}

\begin{abstract}
In this paper, we study the structure of Fano fibrations of varieties admitting an int-amplified endomorphism.
We prove that if a normal $\Q$-factorial klt projective variety $X$ has an int-amplified endomorphism, then there exists an \'etale in codimension one finite morphism $\widetilde{X} \to X$ such that $\widetilde{X}$ is of Fano type over its albanese variety.
As a corollary, if we further assume that $X$ is smooth and rationally connected, then $X$ is of Fano type.
\end{abstract}

\maketitle

\setcounter{tocdepth}{1}

\section{Introduction}
Let $X$ be a normal $\Q$-factorial klt projective variety over an algebraically closed field $k$ of characteristic zero.
We say that a surjective endomorphism $f \colon X \to X$ over $k$ is  \emph{int-amplified} if there exists an ample Cartier divisor $H$ on $X$ such that $f^*H-H$ is ample.
For example, non-invertible polarized endomorphisms are int-amplified.
Admiiting an int-amplified endomorphism imposes strong conditions on the structure of $X$.
Indeed, Nakayama \cite{nakayama02} proved that if $X$ is a smooth rational surface admitting a non-invertible surjective endomorphism, then $X$ is toric.
Recently, Meng \cite{meng17} proved the following theorem.
\begin{thm}[\cite{meng17}, cf. \cite{meng-zhang}]\label{meng's theorem}
We assume that $X$ has an int-amplified endomorphism.
There exists a quasi-\'etale finite cover $\mu \colon \widetilde{X} \to X$, that is, $\mu$ is an \'etale in codimension one finite morphism such that
the albanese morphism $\mathrm{alb}_{\widetilde{X}}$ is a fiber space whose general fiber is rationally connected.
\end{thm}
Following the above results, we discuss the next question in this paper.
\begin{question}[\textup{cf. \cite[Question 6.6]{meng-zhang18-survey}}]\label{main question}
We assume that $X$ has an int-amplified endomorphism.
After replacing with a quasi-\'etale finite cover, 
is a general fiber of the albanese morphism of $X$ toric? 
In particular, if $X$ is smooth and rationally connected, then is $X$ toric?
\end{question}

First, we recall the notion of Fano type.
Given a projective morphism $Z \to B$ of normal varieties, we say that $Z$ is of {\em Fano type} over $B$ if there exists an effective $\Q$-Weil divisor $D$ on $Z$ such that $(Z,D)$ is klt and $-(K_Z+D)$ is ample over $B$ (see $\S$ 2 for the details).
When $B$ is a point, we simply say that $Z$ is of Fano type.
We note that if $Z$ is of Fano type over $B$, then a general fiber is of Fano type.
For example, toric varieties are of Fano type and projective bundles over a variety $B$ are of Fano type over $B$.
Zhang \cite{zhang06} and Hacon-Mckernan \cite{hacon-mackernan07} proved that varieties of Fano type are rationally connected.
On the other hand, smooth and rationally connected varieties are not necessarily of Fano type in general. 
Hence the following theorem strengthens Theorem \ref{meng's theorem} and gives a partial answer to Question \ref{main question}.

\begin{thm}[Theorem \ref{main thm'}]\label{main thm}
We assume that $X$ has an int-amplified endomorphism.
There exists a quasi-\'etale finite cover $\mu \colon \widetilde{X} \to X$ such that
the albanese morphism $\mathrm{alb}_{\widetilde{X}} \colon \widetilde{X} \to A$ is a fiber space and $\widetilde{X}$ is of Fano type over $A$.
\end{thm}
Furthermore, if $X$ is smooth and rationally connected, then $\widetilde{X}$ has to coincide with $X$ and $A$ has to coincide with a point.
Hence, as a corollary of Theorem \ref{main thm}, we obtain the following result, which gives an affirmative answer to \cite[Conjecture 1.2]{broustet-gongyo17} in the smooth and rationally connected case.

\begin{cor}[Corollary \ref{main cor'}]\label{main cor}
We assume that $X$ has an int-amplified endomorphism.
If $X$ is smooth and rationally connected, then it is of Fano type.
\end{cor}

We briefly explain how to prove Theorem \ref{main thm}.
First suppose that $K_X$ is not pseudo-effective.
Running a minimal model program (MMP, for short) for $X$, we obtain a birational map $\sigma_0 \colon X \dashrightarrow X'$ and a Mori fiber space $\pi_0 \colon X' \to X_1$.
Then we construct an effective $\Q$-Weil divisor $\Delta_1$ on $X_1$ as follows, 
\[
\ord_E(\Delta_1)=\frac{m_E-1}{m_E}
\]
for any prime divisor $E$ on $X_1$, where $m_E$ is a positive integer satisfying $\pi_0^*E=m_E F$ for some prime divisor $F$ on $X'$.
Next, we further assume that $K_{X_1}+\Delta_1$ is not pseudo-effective.
Running an MMP for $(X_1,\Delta_1)$, we obtain a birational map $\sigma_1 \colon X_1 \dashrightarrow X'_1$ and a $(K_{X_1}+\Delta_1)$-Mori fiber space $\pi_1 \colon X_1' \to X_2$.
Then we construct an effective $\Q$-Weil divisor $\Delta_2$ on $X_2$ as follows, 
\[
\ord_E(\Delta_2)=\frac{m_E-1+\ord_F(\Delta_1')}{m_E}
\]
for any prime divisor $E$ on $X_2$, where $F$ is a prime divisor on $X_1'$ satisfying $\pi_1^*E=m_E F$ with positive integer  $m_E$.
Repeating such a process, we obtain the following sequence of rational maps and morphisms
\[
\xymatrix@R=15pt{
X \ar@{-->}[r]^-{\sigma_0} & X' \ar[d]^-{\pi_0} & & &  &    \\
& (X_1,\Delta_1) \ar@{-->}[r]^-{\sigma_1} & (X'_1,\Delta'_1) \ar[d]^-{\pi_1} & &  &   \\
& & (X_2,\Delta_2) \ar@{-->}[r]^-{\sigma_1} & \ar@{}[dr]|\ddots &  & \\
&  &  &      & \ar@{-->}[r]^-{\sigma_r}  &  (X_r',\Delta_r') \ar[d]^-{\pi_r} \\
&  &   &   &   & (W,\Delta_W), \\
}
\]
where $K_W+\Delta_W$ is pseudo-effective.
If $K_X$ or $K_{X_1}+\Delta_1$ is pseudo-effective, we define $(W,\Delta_W)$ as $(X,0)$ or $(X_1,\Delta_1)$, respectively.
After iterating $f$, we prove that there exist an $f$-equivarient birational map $X \dashrightarrow Y$ and a sequence of Mori fiber spaces from $Y$ to $W$ such that the following diagram commutes
\[
\xymatrix@R=15pt{
X \ar@{-->}[r]^-{\sigma_0} & X' \ar[d]^-{\pi_0} \ar@{-->}[rrrr] & &  &  & Y \ar[d]   \\
& (X_1,\Delta_1) \ar@{-->}[r]^-{\sigma_1} & (X'_1,\Delta'_1) \ar[d]^-{\pi_1}  & & &  Y_1 \ar[d] \\
& & (X_2,\Delta_2) \ar@{-->}[r]^-{\sigma_1}  & \ar@{}[dr]|{\ddots} & & \vdots \ar[d] \\
&  &  & & \ar@{-->}[r]^-{\sigma_r}  &  (X_r',\Delta_r') \ar[d]^-{\pi_r} \\
&  & &  & &  (W,\Delta_W). \\
}
\]

Since the above rational maps and morphisms are $f$-equivariant, $W$ has an int-amplified endomorphism $h$ and $R_{\Delta_W} :=R_h+\Delta_W-h^*\Delta_W $ is an effective divisor
(see Proposition \ref{mfs of cbf type for f-pairs}), where $R_h$ is the ramification divisor of $h$. 
The effectivity of $R_{\Delta_W}$ implies that $-(K_W+\Delta_W)$ is pseudo-effective (see \cite{meng17}), hence $K_W+\Delta_W$ is $\Q$-linearly trivial.
Then we prove that $W$ has a finite cover by an abelian variety $A$ (see Proposition \ref{convering theorem}).
Moreover we can lift this cover to $X$ as follows,
\[
\xymatrix{
\widetilde{X} \ar[d]_\mu \ar@{-->}[r]^{\widetilde{\pi}} & A \ar[d] \\
X \ar@{-->}[r]_{\pi} & W,
}
\]
where $\mu$ is a quasi-\'etale finite morphism (see Theorem \ref{lifting of cover of abelian variety}),
and in particular, $\widetilde{\pi}$ and $\pi$ are morphisms.

Finally, we prove that $\widetilde{X}$ is of Fano type over $A$.
Note that $Y$ is of Fano type over $W$ (see Theorem \ref{construction of a towers of mfs}).
Since beeing of Fano type over $W$ is invariant under every equivarient birational map with respect to an int-amplified endomorphism (see Proposition \ref{equivariant birational map preserves fano type}), $X$ is also of Fano type over $W$.
Moreover, since $\mu$ is quasi-\'etale, $\widetilde{X}$ is also of Fano type over $A$.
In conclusion, we obtain Theorem \ref{main thm}.

By the construction of $\widetilde{X}$, we can show that every surjective endomorphism of $X$ lifts to $\widetilde{X}$. 

\begin{thm}[Theorem \ref{main thm'}]\label{main thm; lifting}
We assume that $X$ has an int-amplified endomorphism.
For every surjective endomorphism $\phi$ of $X$,
there exists a following commutative diagram
\[
\xymatrix{
X  \ar[d]_{\phi^m}& \widetilde{X} \ar[l]_\mu \ar[r]^{\mathrm{alb}_{\widetilde{X}}} \ar[d]_{\widetilde{\phi}} & A \ar[d]_{\phi_A} \\
X & \widetilde{X} \ar[l]^\mu \ar[r]_{\mathrm{alb}_{\widetilde{X}}} & A \\
}
\]
for some positive integer $m$, where $\mu$, $\mathrm{alb}_{\widetilde{X}}$, $\widetilde{X}$, $A$ are in Theorem \ref{main thm}, $\widetilde{\phi}$ and $\phi_A$ are surjective endomorphisms.
\end{thm}

\vspace{15pt}
{\bf Notation and Terminology.} \\
Throughout this paper, all varieties difined over an algebraically closed field $k$ of characteristic zero and every morphism of varieties is a morphism over $k$. We will freely use the standard notations in \cite{kollar-mori}.
\begin{itemize}
\item A morphism $f \colon X \to X$ from a projective variety $X$ to itself is called self-morphism of $X$ or endomorphism of $X$.
If it is surjective, then it is a finite morphism.
\item A morphism $f \colon X \to Y$  between varieties is called quasi-\'etale if $f$ is \'etale at every codimension one point of $X$.
\item Let $f \colon X \to Y$ be a finite surjective morphism between normal varieties. The ramification divisor of $f$ is denoted by $R_{f}$.
\item The function field of a variety $X$ is denoted by $K(X)$.
\item Let $f \colon X \to X$ be an endomorphism of a variety $X$.  A subset $S\subset X$ is called totally invariant under $f$ if $f^{-1}(S)=S$ as sets.
\item Consider the commutative diagram
\[
\xymatrix{
X \ar@{-->}[r]^{\pi} \ar[d]_{f} & Y \ar[d]^{g} \\
X \ar@{-->}[r]_{\pi} & Y,
}
\]
where $f, g$ are surjective morphisms and $\pi$ is a dominant rational map.
We write this diagram as 
\[
\xymatrix{
f \acts X \ar@{-->}[r]^{\pi}& Y \racts g.
}
\]
We say a commutative diagram is equivariant if each object is equipped with an endomorphism and the morphisms are equivariant with respect to these 
endomorphisms.
\end{itemize}

\begin{ack}
The author wishes to express his gratitude to his supervisor Professor Shunsuke Takagi for his encouragement, valuable advice and suggestions. He is also
grateful to Dr.~Kenta Sato, Dr.~Kenta Hashizume, Dr.~Yohsuke Matsuzawa, Dr.~Sho Ejiri, Dr.~Masaru Nagaoka, Professor Yoshinori Gongyo, Professor Sheng Meng and Profesor De-Qi Zhang for their helpful comments and suggestions. 
This work was supported by the Program for Leading Graduate Schools, MEXT, Japan.
\end{ack}

\section{Preliminaries}

\subsection{Varieties of Fano type and Calabi--Yau type}
In this paper, we use the following terminology.

\begin{defn}[{cf.~\cite[Definition 2.34]{kollar-mori}, \cite[Remark 4.2]{schwede-smith}}]
Let $X$ be a normal variety and $\Delta$ be an effective $\Q$-Weil divisor on $X$ such that $K_X+\Delta$ is $\Q$-Cartier.
Let $\pi \colon Y \to X$ be a birational morphism from a normal variety $Y$. 
Then we can write
\[
K_Y=\pi^*(K_X+\Delta)+\sum_{E}(a_E(X,\Delta)-1)E,
\]
where $E$ runs through all prime divisors on $Y$.
We say that the pair $(X,\Delta)$ is {\em log canonical} or {\em lc}, for short (resp.,~{\em Kawamata log terminal} or {\em klt}, for short) if $a_E(X,\Delta) \geq 0$ (resp., $a_E(X,\Delta)>0$) for every prime divisor $E$ over $X$. 
If $\Delta=0$, we simply say that $X$ is log canonical (resp., klt).
\end{defn}

\begin{defn}[cf. {\cite[Lemma-Definition 2.6]{prokhorov-shokurov}}]
Let $\pi \colon X \to B$ be a projective morphism of normal varieties and $\Delta$ be an effective $\Q$-Weil divisor on $X$.
\begin{enumerate}
    \item We say that $(X,\Delta)$ is {\em log Fano over $B$} if $-(K_X+\Delta)$ is $\pi$-ample $\Q$-Cartier and $(X,\Delta)$ is klt.
    We say that $(X,\Delta)$ is of {\em Fano type over $B$} if there exists an effective $\Q$-Weil divisor $\Gamma$ on $X$ such that $(X,\Delta+\Gamma)$ is log Fano over $B$.
    If $B$ is a point, we simply say that $(X,\Delta)$ is of Fano type.
    \item We say that $(X,\Delta)$ is {\em log Calabi--Yau over $B$} if $K_X+\Delta \sim_{\Q, B} 0$ and $(X,\Delta)$ is log canonical.
    We say that $(X,\Delta)$ is of {\em Calabi--Yau type over $B$} if there exists an effective $\Q$-Weil divisor $\Gamma$ on $X$ such that $(X,\Delta+\Gamma)$ is log Calabi--Yau over $B$.
    If $B$ is a point, we say that $(X,\Delta)$ is of Calabi-Yau type.
\end{enumerate}

\end{defn}

We introduce the following basic properties.

\begin{prop}\label{first property of Fano type}
Let $\pi \colon X \to B$ be a surjective projective morphism of normal projective varieties and $\Delta$ an effective $\Q$-Weil divisor on $X$.
Then the following conditions are equivalent to each other.
\begin{enumerate}
    \item $(X,\Delta)$ is of Fano type over $B$.
    \item There exists an effective $\pi$-big $\Q$-Weil divisor $\Omega$, that is, $\Omega$ is a sum of an $\pi$-ample $\Q$-Cartier divisor and an effective $\Q$-Weil divisor such that $(X,\Delta+\Omega)$ is klt and $K_X+\Delta+\Omega \sim_{\Q,B} 0$.
    \item There exists an effective $\pi$-big $\Q$-Weil divisor $\Omega$ such that $(X,\Delta+\Omega)$ is klt and $K_X+\Delta+\Omega \equiv_{B} 0$.
\end{enumerate}
\end{prop}

\begin{proof}
First we assume that $(X,\Delta)$ is of Fano type over $B$.
Then there exists an effective $\Q$-Wei divisor $\Gamma$ such that $(X,\Delta+\Gamma)$ is klt and $-(K_X+\Delta+\Gamma)$ is $\pi$-ample. 
Then we can take an effective $\Q$-Weil divisor $\Omega'$ which is $\Q$-linear equivalent to $-(K_X+\Delta+\Gamma)$ over $B$ such that $(X,\Delta+\Gamma+\Omega')$ is klt.
Therefore, $\Omega=\Gamma+\Omega' \sim_{\Q,B} -(K_X+\Delta)$ is big over $B$, $K_X+\Delta+\Omega \sim_{\Q,B} 0$ and $(X,\Delta+\Omega)$ is klt.

It is clear that the second condition implies the third condition.

Next we assume that we can take $\Omega$ satisfying the third condition.
Then there exist a $\pi$-ample $\Q$-Cartier divisor $A$ and an effective $\Q$-Weil divisor $D$ such that $\Omega = A + D$.
Since $(X,\Delta+\Omega)$ is klt, for enough small $\epsilon>0$ such that $(X,\Delta+(1-\epsilon)\Omega+\epsilon D)$ is klt and
\[
K_X+\Delta+(1-\epsilon)\Omega+\epsilon D \equiv_{B} -\epsilon A
\]
is anti-ample over $B$.
It means that $(X,\Delta)$ is of Fano type over $B$.
\end{proof}

\begin{prop}\label{fano type birational contraction}
We consider the following commutative diagram
\[
\xymatrix{
X \ar@{-->}[rr]^-{\mu} \ar[rd]_-{\pi_X}&  &  Y \ar[ld]^{\pi_Y} \\
& B, & \\
}
\]
where $X, Y, B$ are normal projective varieties, $\pi_X, \pi_Y$ are surjective projective morphisms and $\mu$ is a birational contractin, that is, $\mu^{-1}$ has no exceptional divisors.
Let $\Delta$ be an effective $\Q$-Weil divisor on $X$ such that $(X,\Delta)$ is of Fano type over $B$ and $\Delta'=\mu_*\Delta$.
Then $(Y,\Delta')$ is also of Fano type over $B$.

\end{prop}

\begin{proof}
By Proposition \ref{first property of Fano type}, there exists an effective $\pi_X$-big $\Q$-Weil divisor $\Omega$ on $X$ such that $(X,\Delta+\Omega)$ is klt and $K_X+\Delta+\Omega \sim_{\Q,B} 0$.
Let $\Omega' =\mu_*\Omega$, then $\Omega'$ is $\pi_Y$-big $\Q$-Weil divisor.
Indeed, since $\Omega$ is $\pi_X$-big, we have $\Omega = A+D$, where $A$ is an $\pi_X$-ample $\Q$-Cartier divisor and $D$ is an effective $\Q$-Weil divisor.
Taking an $\pi_Y$-ample Cartier divisor $H$ on $Y$, there exists an effective divisor $A' \sim_{\Q,B} A$ such that $mA' \geq \mu^{-1}_*H$ for some positive integer $m$.
Hence, we have $\Omega' \sim_{\Q,B} \mu_*(A'+D) \geq \frac{1}{m}H $, and in particular, $\Omega'$ is $\pi_Y$-big.
Furthermore, $K_Y+\Delta'+\Omega'$ is $\Q$-Cartier and $K_Y+\Delta'+\Omega' \sim_{\Q,B}0$.
By the negativity lemma, $(Y,\Delta'+\Omega')$ is klt.
\end{proof}

The reader is referred to \cite[Lemma-Definition 2.6]{prokhorov-shokurov} for more details. 

\subsection{Int-amplified endomorphism}
Meng and Zhang established minimal model program equivariant with respect to endomorphisms in \cite{meng-zhang2}, for varieties admitting an int-amplified endomorphism.
We summarize their results that we need later.

\begin{defn}\label{def:intamp}
A surjective endomorphism $f \colon X \to X$ of normal projective variety $X$ is called \emph{int-amplified} if 
there exists an ample Cartier divisor $H$ on $X$ such that $f^{*}H-H$ is ample.
\end{defn}

We collect basic properties of int-amplified endomorphisms in the following lemma.

\begin{prop}\label{first properties of int-ampl endo}\ 
\begin{enumerate}
\item Let $X$ be a normal projective variety, $f \colon X \to X$ a surjective morphism, and $n>0$ a positive integer.
Then, $f$ is int-amplified if and only if so is $f^{n}$.

\item Let $\pi \colon X \to Y$ be a surjective morphism between normal projective varieties.
Let $f \colon X \to X$, $g \colon Y \to Y$ be surjective endomorphisms such that $\pi \circ f=g \circ \pi$.
If $f$ is int-amplified, then so is $g$. 

\item  Let $\pi \colon X \dashrightarrow Y$ be a dominant rational map between normal projective varieties of same dimension.
Let $f \colon X \to X$, $g \colon Y \to Y$ be surjective endomorphisms such that $\pi \circ f=g \circ \pi$.
Then $f$ is int-amplified if and only if so is $g$.

\item Let $f \colon X \to X$ be an int-amplified endomorphism of a normal projective variety and $D$ a $\Q$-Cartier divisor on $X$.
If $f^*D-D$ is numerically equivalent to an effective $\Q$-Weil divisor, then $D$ is also numerically equivalent to an effective $\Q$-Weil divisor.
In particular, if $X$ is $\Q$-Gorenstein, then $-K_X$ is numerically equivalent to an effective $\Q$-Weil divisor.

\end{enumerate}
\end{prop}
\begin{proof}
See \cite[Theorem 3.3, Lemma 3,4, 3.5, Theorem 1.5]{meng17}.
\end{proof}

\begin{thm}[Meng-Zhang]\label{equivariant extremal ray contraction}
Let $X$ be a $\Q$-factorial normal projective variety admitting an int-amplified endomorphism.
Let $ \Delta$ be an effective $\Q$-Weil divisor on $X$ such that $(X, \Delta)$ is klt.
\begin{enumerate}
\item\label{fin-st} There are only finitely many $(K_{X}+ \Delta)$-negative extremal rays of $ \overline{NE}(X)$.
Moreover, let $f \colon X \to X$ be a surjective endomorphism of $X$.
Then every $(K_{X}+ \Delta)$-negative extremal ray is fixed by the linear map $(f^{n})_{*}$ for some $n>0$.
\item\label{end-induced}  Let $f \colon X \to X$ be a surjective endomorphism of $X$.
Let $R$ be a $(K_{X}+ \Delta)$-negative extremal ray and $\pi \colon X \to Y$ its contraction.
Suppose $f_{*}(R)=R$.
Then,
\begin{enumerate}
\item \label{induced-on-contr} $f$ induces an endomorphism $g \colon Y \to Y$ such that $g\circ \pi=\pi \circ f$;
\item \label{induced-on-flip}if $\pi$ is a flipping contraction and $X^{+}$ is the flip, the induced rational self-map $h \colon X^{+} \dashrightarrow X^{+}$
is a morphism.
\end{enumerate}
\item  In particular, for any finite sequence of $(K_{X}+ \Delta)$-MMP and for any surjective endomorphism $f \colon X \to X$,
there exists a positive integer $n>0$ such that the sequence of MMPs is equivariant under $f^{n}$.
\end{enumerate}
\end{thm}
\begin{proof}
(\ref{fin-st}) is a special case of  \cite[Theorem 4.6]{meng-zhang2}.
(\ref{induced-on-contr}) is true since the contraction is determined by the ray $R$.
(\ref{induced-on-flip}) follows from \cite[Lemma 6.6]{meng-zhang}.
\end{proof}

\begin{thm}[Equivariant MMP (Meng-Zhang)]\label{equiMMP}
Let $X$ be a $\Q$-factorial klt projective variety admitting an int-amplified endomorphism.
Then for any surjective endomorphism $f \colon X \to X$, there exists a positive integer $n>0$ and 
a sequence of rational maps
\begin{align*}
X=X_{0} \dashrightarrow X_{1} \dashrightarrow \cdots \dashrightarrow X_{r}
\end{align*}
such that
\begin{enumerate}
\item $X_{i} \dashrightarrow X_{i+1}$ is either a divisorial contraction, flip, or Mori fiber space of a 
$K_{X_{i}}$-negative extremal ray,
\item there exist surjective endomorphisms $g_{i} \colon X_{i} \to X_{i}$ for $i=0, \dots ,r$
such that $g_{0}=f^{n}$ and the following diagram commutes
\[
\xymatrix{
X_{i} \ar@{-->}[r] \ar[d]_{g_{i}} & X_{i+1} \ar[d]^{g_{i+1}}\\
X_{i} \ar@{-->}[r] & X_{i+1},
}
\]
\item $X_{r}$ is a $Q$-abelian variety (that is, there exists a quasi-\'etale finite surjective morphism $A \to X_{r}$ from an abelian variety $A$,
note that $X_{r}$ might be a point).
In this case, there exists a quasi-\'etale finite surjective morphism $A \to X_{r}$ from an abelian variety $A$
and an surjective endomorphism $h \colon A \to A$ such that the diagram
\[
\xymatrix{
A \ar[r]^{h} \ar[d] & A \ar[d]\\
X_{r} \ar[r]_{g_{r}} & X_{r} 
}
\]
commutes.
\end{enumerate}
\end{thm}
\begin{proof}
This is a part of \cite[Theorem 1.2]{meng-zhang2}.
\end{proof}

\begin{rmk}
Surjective endomorphisms of a Q-abelian variety always lift to a certain quasi-\'etale cover by an abelian variety.
See \cite[Lemma 8.1 and Corollary 8.2]{cascini-meng-zhang}, for example. The proof works over any algebraically closed field.
\end{rmk}

\section{Pairs with respect to an endomorphism}
In this section, we study equivariant Mori fiber spaces.
First we introduce the notion of the pair with reapect to surjective endomorphisms.
In Proposition 3.8, we construct such a pair on base varieties of equivariant Mori fiber spaces.

\begin{defn}\label{definition of f-pair}
Let $f \colon X \to X$ be a surjective endomorphism of a normal variety $X$.
Then $(X,\Delta)$ is called an {\em $f$-pair} if
\begin{enumerate}
    \item $\Delta$ is an effective $\Q$-Weil divisor, and
    \item $R_{\Delta}: = R_f+\Delta-f^*\Delta \geq 0$.
\end{enumerate}
\end{defn}

\begin{rmk}
Let $f \colon X \to X$ be a surjective endomorphism of a normal variety $X$.
\begin{itemize}
    \item $(X,0)$ is an $f$-pair.
    \item If $(X,\Delta)$ is an $f$-pair, then
\[
R_{\Delta}=R_f+\Delta-f^*\Delta \sim K_X+\Delta-f^*(K_X+\Delta).
\]
Furthermore, for every positive integer $m$, we have
\[
R_{\Delta,m+1} :=R_{f^{m+1}}+\Delta-(f^{m+1})^*\Delta=R_\Delta+f^*R_{\Delta,m}.
\]
In particular, $(X,\Delta)$ is an $f^m$-pair for all $m$.
\end{itemize}
\end{rmk}

\begin{eg}\label{example; f-pair}
Let $E$ be an elliptic curve and $[m]$ a multiplication by $m$ for all integers $m$.
Since $[m]$ is $[-1]$-equivariant, we obtain the following commutative diagram
\[
\xymatrix{
E \ar[r]^\mu \ar[d]_{[m]} & \P^1 \ar[d]^h \\
E  \ar[r]^\mu & \P^1,
}
\]
where $\mu$ is the quotient map by $[-1]$ and $h$ is the endomorphism induced by $[m]$.
Let $Q_1, \ldots , Q_4$ be the $2$-torsion points on $E$ and $P_1 =\mu(Q_1), \ldots , P_4=\mu(Q_4)$.
Let 
\[
\Delta=\frac{1}{2}(P_1+P_2+P_3+P_4),
\]
then $(\P^1,\Delta)$ is an $h$-pair (see Example \ref{example; mfs of cbf} and Example \ref{example; quasi-etale cover}).
\end{eg}

\begin{prop}\label{first property of f-pairs}
Let $f \colon X \to X$ be an int-amplified endomorphism of a normal projective variety $X$ and $(X,\Delta)$ an $f$-pair such that $K_X+\Delta$ is $\Q$-Cartier.
Then $-(K_X+\Delta)$ is numerically equivalent to an effective $\Q$-Weil divisor.

\end{prop}

\begin{proof}
It follows from Proposition \ref{first properties of int-ampl endo} and
\[
0 \leq R_{\Delta} \sim K_X+\Delta-f^*(K_X+\Delta) 
\]
\end{proof}

\begin{rmk}
Let $f \colon X \to X$ be an int-amplified endomorphism of a normal projective variety $X$.
Then an $f$-pair $(X,\Delta)$ is valuative log canonical defined in \cite{yoshikawa18}, and the proof is similar to the proof of \cite[Theorem 1.4]{yoshikawa18}.
In particular, if $K_X+\Delta$ is $\Q$-Cartier, then $(X,\Delta)$ is log canonical.
\end{rmk}

In order to introduce very important $f$-pairs, we define the following fiber spaces of pairs.

\begin{defn}\label{definition of mfs of cbf type}
A morphism $\pi \colon (X,\Delta) \to (Y,\Gamma)$ of pairs is called a {\em Mori fiber space of canonical bundle formula type} if
\begin{enumerate}
    \item $X$ is a normal $\Q$-factorial projective  variety and $\Delta$ is an effective $\Q$-Weil divisor on $X$ such that $(X,\Delta)$ is klt,
    \item $\pi \colon X \to Y$ is a $(K_X+\Delta)$-Mori fiber space, and
    \item for any prime divisor $E$ on $Y$, $\Gamma$ satisfies
    \[
    \ord_{E}(\Gamma)=\frac{m_E-1+\ord_{F}(\Delta)}{m_E},
    \]
    where $F$ is a prime divisor on $X$ satisfying $\pi^*E=m_E F$ with  positive integer $m_E$.
\end{enumerate}
\end{defn}

\begin{rmk} \ 
\begin{itemize}
    \item If $Y$ is a point, the third condition is always satisfied.
    \item Since $\pi$ is a $(K_X+\Delta)$-Mori fiber space, we can take $m_E$ and $F$ as in Definition \ref{definition of mfs of cbf type} (see \cite[Lemma 4.10]{yoshikawa19}).
\end{itemize}
\end{rmk}

\begin{prop}\label{pairs of cbf type are klt}
Let $\pi \colon (X,\Delta) \to (Y,\Gamma)$ be a Mori fiber space of canonical bundle formula type.
Then $(Y,\Gamma)$ is klt.
\end{prop}

\begin{proof}
Since $-(K_X+\Delta)$ is $\pi$-ample, we can take an ample Cartier divisor $H$ on $Y$ such that $-(K_X+\Delta)+\pi^*H$ is ample $\Q$-Cartier.
There exists $0 \leq D\sim_{\Q} -(K_X+\Delta)+\pi^*H$ such that $(X,\Delta+D)$ is klt.
By the Ambro's canonical bundle formula  \cite[Theorem 4.1]{ambro05},
$(Y,B)$ is klt, where $B$ is an effective $\Q$-Weil divisor satisfying 
\[
\ord_{E}(B)=1-\mathrm{lct}_{\eta_E}(X,\Delta+D ; \pi^*E)
\]
for any prime divisor $E$ with generic point $\eta_E$.
Let $\pi^*E=m_E F$ for some positive integer $m_E$ and prime divisor $F$, then we have
\[
\mathrm{lct}_{\eta_E}(X,\Delta+D;\pi^*E) \leq \frac{1-\ord_{F}(\Delta+D)}{m_E} \leq \frac{1-\ord_{F}(\Delta)}{m_E}.
\]
In particular, we have
\[
\ord_{E}(B) \geq \frac{m_E-1+\ord_F(\Delta)}{m_E} = \ord_E(\Gamma).
\]
It means that $B \geq \Gamma$, so $(Y,\Gamma)$ is also klt.
\end{proof}

The following prposition gives a very important pairs.

\begin{prop}\label{mfs of cbf type for f-pairs}
We consider the following commutative diagram
\[
\xymatrix{
(X,\Delta) \ar[r]^{\pi} \ar[d]_{f} & (Y,\Gamma) \ar[d]^{g} \\
(X,\Delta) \ar[r]_{\pi}  & (Y,\Gamma), 
}
\]
where $\pi$ is a Mori fiber space of canonical bundle formula type, $f$, $g$ are surjective endomorphisms and $(X,\Delta)$ is an $f$-pair.
Then
\begin{itemize}
    \item $(Y,\Gamma)$ is a $g$-pair, and
    \item $R_{\Delta}-\pi^*R_{\Gamma}$ is effective and has no vertical components of $\pi$, that is, for every prime divisor $F$ on $X$ with $\pi(F)\neq Y$, we have $\ord_F(R_{\Delta}-\pi^*R_{\Gamma})=0$.
\end{itemize}
\end{prop}

\begin{proof}
We take a prime divisor $E$ on Y, then we have $\pi^*E=m_E F$ for some positive integer $m_E$ and prime divisor $F$ on $X$.
If the second assertion holds, then
\[
m_E \ord_{E}(R_{\Gamma})=\ord_F(\pi^*R_{\Gamma})=\ord_F(R_{\Delta}) \geq 0,
\]
so the first assertion holds.
Therefore, it is enough to show the second assertion.
Let 
\[
g^*E=a_1E_1+ \cdots +a_rE_r,
\]
and
\[
f^*F=b_1F_1+ \cdots +b_rF_r,
\]
where all $a_i$, $b_i$ are positive integers and $E_i$, $F_i$ are prime divisors with $\pi^*E_i=m_{E_i}F_i$ for some positive integer $m_{E_i}$.
Since $\pi^*g^*E=f^*\pi^*E$, we have $m_E b_i=a_i m_{E_i}$ for all $i$.
Therefore, we have
\begin{eqnarray*}
    \ord_{E_i}(R_{\Gamma})&=&\ord_{E_i}(R_g+\Gamma-g^*\Gamma) \\
    &=& a_i-1+\frac{m_{E_i}-1+\ord_{F_i}(\Delta)}{m_{E_i}}-a_i \frac{m_E-1+\ord_F(\Delta)}{m_E} \\
    &=& \frac{-1+\ord_{F_i}(\Delta)+b_i-b_i\ord_F(\Delta)}{m_{E_i}} \\
    &=& \frac{\ord_{F_i}(R_\Delta)}{m_{E_i}},
\end{eqnarray*}
it means that $\ord_{F_i}(R_{\Delta}-\pi^*R_{\Gamma})=0$.
\end{proof}

\begin{eg}\label{example; mfs of cbf}
In \cite[Section 7]{matsuzawa-yoshikawa19-surface}, we give the following commutative diagram
\[
\xymatrix{
g \acts Y \ar[r]^-{\widetilde{\pi}} \ar[d]_{\widetilde{\mu}} & E \racts [m] \ar[d]^\mu \\
f \acts X \ar[r]_-\pi & \P^1 \racts h,
}
\]
where $Y$ is a ruled surface over $E$, $\widetilde{\mu}$ is quasi-\'etale, $\pi$ is a Mori fiber space, $E, \mu,$ and $h$ are in Example \ref{example; f-pair}.
Then $\pi \colon (X,0) \to (\P^1,\Delta)$ is a Mori fiber space of canonical bundle formula type, where $\Delta$ is in Example \ref{example; f-pair}.
In particular, $(\P^1,\Delta)$ is an $h$-pair by Proposition \ref{mfs of cbf type for f-pairs}.
\end{eg}

\section{Construction of the tower of Mori fiber spaces}
In this section, we study Fano type assuming the existence of int-amplified endomorphisms.
Corollary \ref{fano type of a tower of mfs} means that if the variety has an int-amplified endomorphism and on the tower of Mori fiber spaces of canonical bundle formula type, then it is of Fano type over the bottom variety.
Using this corollary, we replace a sequence of steps of MMP with a tower of Mori fiber spaces of canonical bundle formula type (see Theorem \ref{construction of a towers of mfs}).

\begin{lem}\label{ramification divisor is relative ample}
Let $X$ be a normal $\Q$-factorial projective variety and $\Delta$ an effective $\Q$-Weil divisor.
We consider the following commutative diagram
\[
\xymatrix{
X \ar[r]^{\pi} \ar[d]_{f} & Y \ar[d]^{g} \\
X \ar[r]_{\pi}  & Y, 
}
\]
where $\pi$ is a $(K_X+\Delta)$-negative extremal ray contraction, $f$, $g$ are int-amplified endomorphisms and $(X,\Delta)$ is an $f$-pair.
Then $R_{\Delta}$ is $\pi$-ample.
\end{lem}

\begin{proof}
Since $N^1(X/Y)$ is $1$-dimensional, there exists a positive number $a$ such that $f^*|_{N^1(X/Y)}=a \cdot \id_{N^1(X/Y)}$.
Hence we have
\[
R_{\Delta} \sim K_X+\Delta-f^*(K_X+\Delta) \equiv_{Y} (1-a)(K_X+\Delta).
\]
Since $-(K_X+\Delta)$ is $\pi$-ample, it is enough to show $a>1$.
By the definition of int-amplified endomorphisms, there exists an ample Cartier divisor $H$ on $X$ such that $f^*H-H$ is also ample.
On the othe hand, $f^*H-H \equiv_{Y} (a-1)H$, so we have $a>1$.
\end{proof}

\begin{lem}\label{Fano type if the ramification divisor contains ample div}
We consider the following commutative diagram
\[
\xymatrix{
(X,\Delta) \ar[r]^{\pi} \ar[d]_{f} & (Y,\Gamma) \ar[r]^{b} \ar[d]^{g} & B \ar[d]^{h}  \\
(X,\Delta) \ar[r]_{\pi}  & (Y,\Gamma) \ar[r]_{b} & B, 
}
\]
where $\pi$ is a Mori fiber space of canonical bundle formula type, $f$, $g$, $h$ are int-amplified endomorphisms, $b$ is a surjective morphism to a normal projective variety $B$ and $(X,\Delta)$ is an $f$-pair.
Assume that $Y$ is of Fano type over $B$ and $R_{\Gamma}$ contains an effective $b$-ample $\Q$-Cartier divisor $H$ on $Y$.
Then 
\begin{itemize}
    \item $R_{\Delta}$ contains an effective $(b\circ \pi)$-ample $\Q$-Cartier divisor, and
    \item $(X,\Delta)$ is of Fano type over $B$.
\end{itemize}
\end{lem}

\begin{proof}
By Lemma \ref{ramification divisor is relative ample}, there exists a positive integer $m$ such that
\[
H_X :=\frac{1}{m}(R_{\Delta}-\pi^*H)+\pi^*H=\frac{1}{m}R_{\Delta}+(1-\frac{1}{m})\pi^*H
\]
is $(b\circ \pi)$-ample.
By Proposition \ref{mfs of cbf type for f-pairs}, we have 
\[
0 \leq \pi^*H \leq \pi^*R_{\Gamma} \leq R_{\Delta}.
\]
In particular, we have $0 \leq H_X \leq R_{\Delta}$, and we obtain the first assertion.

Next, we prove the second assertion. 
First, we prove the following claim.
\begin{claim}
There exists a positive integer $m$ such that $-(K_X+\Delta)+(f^m)^*\pi^*H$ is $(b\circ \pi)$-ample.
\end{claim}

\begin{claimproof}
Since $-(K_X+\Delta)$ is $\pi$-ample, there exists a positive integer $k$ such that $-(K_X+\Delta)+k\pi^*H$ is $(b\circ \pi)$-ample.
Since we have
\begin{eqnarray*}
-(K_X+\Delta)+(f^m)^*\pi^*H&=&-(K_X+\Delta)+k\pi^*H+((f^m)^*\pi^*H-k\pi^*H)\\
&=& -(K_X+\Delta)+k\pi^*H+\pi^*((g^m)^*H-kH),
\end{eqnarray*}
it is enough to show that there exists a positive integer $m$ such that $(g^m)^*H-kH$ is $b$-ample.
Since $Y$ is of Fano type over $B$, the number of the extremal rays of $\NE(Y/B)$ is finite by the cone theorem (cf. \cite[Theorem 3.7]{kollar-mori}).
Let $R_1, \ldots , R_N$ be all extremal rays of $\NE(Y/B)$.
Replacing $g$ by some iterate of $g$, we may assume that $g_*R_i=R_i$ for all $i$. 
Let $v_i \in R_i \backslash \{0\}$,
then $g_*(v_i)=a_iv_i$ for some positive number $a_i$ for all $i$.
By the definition of int-amplified endomorphisms, there exists an ample Cartier divisor $A$ on $Y$ such that $g^*A-A$ is ample.
Then we have
\[
0 < (g^*A-A)\cdot v_i=(a_i-1)(A \cdot v_i),
\]
so we obtain $a_i>1$ for all $i$.
In particular, for enough large $m$, we have
\[
((g^m)^*H-kH)\cdot v_i=(a_i^m-k)(H\cdot v_i) >0,
\]
and it means that $(g^m)^*H-kH$ is $b$-ample.
\end{claimproof}

By the above claim, we may assume that $-(K_X+\Delta)+f^*\pi^*H$ is $(b\circ \pi)$-ample, replacing $f$ by some iterate of $f$.
In particular, there exists a Weil divisor $D\sim_{\Q,B}-(K_X+\Delta)+f^*\pi^*H$ such that $(X,\Delta+D)$ is klt.
Next, we define a $\Q$-Weil divisor $D_1$ as 
\[
D_1=d^{-1}f_*(R_f+\Delta+D)-\Delta,
\]
where $d=\deg(f)$.
Then we have
\begin{eqnarray*}
D_1=d^{-1}f_*(R_f+\Delta+D)-\Delta&=&d^{-1}f_*(R_{\Delta}+f^*\Delta+D)-\Delta \\
&=& d^{-1}f_*(R_{\Delta}+D) \geq 0,
\end{eqnarray*}
and
\[
D_1+\Delta=d^{-1}f_*(R_f+\Delta+D).
\]
Since this construction is the same as the construction in \cite[Lemma 1.1]{fujino-gongyo}, $(X,\Delta+D_1)$ is klt and we have
\[
f^*(K_X+\Delta+D_1)\equiv K_X+\Delta+D \equiv_B f^*\pi^*H
\]
by an argument similar to the proof of \cite{fujino-gongyo}.
In particular, we have
\[
K_X+\Delta+D_1\equiv_B \pi^*H.
\]
Next, we construct $D_2$ by the same way, that is,
\[
D_2=d^{-1}f_*(R_f+\Delta+D_1)-\Delta=d^{-1}f_*(R_{\Delta}+D_1).
\]
Then $D_2$ is effective, $(X,\Delta+D_2)$ is klt and
\[
K_X+\Delta+D_2\equiv_B d^{-1}f_*\pi^*H
\]
by the same way as above.
By the construction, we have
\[
D_2 \geq d^{-1}f_*R_{\Delta} \geq d^{-1}f_*\pi^*H.
\]
Let $\Omega=D_2-d^{-1}f_*\pi^*H$, then $(X,\Delta+\Omega)$ is klt and we have
\[
K_X+\Delta+\Omega\equiv_B d^{-1}f_*\pi^*H-d^{-1}f_*\pi^*H=0.
\]
Therefore it is enough to show that $-(K_X+\Delta)$ is $(b\circ \pi)$-big by Proposition \ref{first property of Fano type}.
Since $R_{\Delta}$ contains $(b\circ \pi)$-ample $\Q$-Cartier divisor, $R_{\Delta}$ is $(b\circ \pi)$-big.
Since $-(K_X+\Delta)$ is pseudo-effective by Proposition \ref{first property of f-pairs}, 
\[
-f^*(K_X+\Delta)\sim R_{\Delta}-(K_X+\Delta)
\]
is also $(b\circ \pi)$-big.
In particular, $-(K_X+\Delta)$ is $(b\circ \pi)$-big, and we obtain Lemma \ref{Fano type if the ramification divisor contains ample div}.
\end{proof}

\begin{cor}\label{fano type of a tower of mfs}
We consider the following sequence
\[
\xymatrix{
(X_0,\Delta_0) \ar[r]^-{\pi_0} & (X_1,\Delta_1) \ar[r]^-{\pi_1} & \cdots \ar[r]^-{\pi_{r-1}} &(X_r,\Delta_r) \ar[r]^-{\rho} & B, 
}
\]
where
\begin{itemize}
    \item $\pi_i \colon (X_i,\Delta_i) \to (X_{i+1},\Delta_{i+1})$ is a Mori fiber space of canonical bundle formula type for $0 \leq i \leq r-1$,
    \item $\rho \colon X_r \to B$ is a $(K_{X_r}+\Delta_r)$-negative extremal ray contraction, and
    \item $X_0$ has an int-amplified endomorphism $f_0$ such that $(X_0,\Delta_0)$ is an $f_0$-pair.
\end{itemize}
Then $(X_0,\Delta_0)$ is of Fano type over $B$.
\end{cor}

\begin{proof}
By Theorem \ref{equivariant extremal ray contraction}, we may assume that the above sequence is $f_0$-equivarient, replacing $f_0$ by some iterate.
By Proposition \ref{mfs of cbf type for f-pairs}, $(X_i,\Delta_i)$ is an $f_i$-pair, where $f_i$ is the endomorphism of $X_i$ induced by $f_0$.
By Lemma \ref{ramification divisor is relative ample}, $R_{\Delta_r}$ is $\rho$-ample.
Since $-(K_{X_r}+\Delta_r)$ is $\rho$-ample and $(X_r,\Delta_r)$ is klt,
$(X_r,\Delta_r)$ is of Fano type over $B$.
By using Lemma \ref{Fano type if the ramification divisor contains ample div} inductively, we obtain $R_{\Delta_i}$ contains an effective $\Q$-Cartier divisor which is ample over $B$ and $(X_i,\Delta_i)$ is of Fano type over $B$ for all $i$.
In conclusion, $(X_0,\Delta_0)$ is of Fano type over $B$.
\end{proof}

\begin{lem}\label{two-ray game}\textup{\cite[Lemma 3.3,  3.4]{cerbo-svaldi}}
Let $X$ be a normal $\Q$-factorial projective variety and $\Delta$ an effective $\Q$-Weil divisor.
Let $\pi \colon X \to Y$ be a Mori fiber space and $\rho \colon Y \to W$ an extremal ray contraction which is birational.
Assume that $(X,\Delta)$ is of Fano type over $W$.
Then we have the following commutative diagram
\[
\xymatrix{
X \ar@{-->}[rr]^-{\mu_X} \ar[d]_-{\pi} &   & X' \ar[d]^-{\pi'} \\
Y \ar@{-->}[rr]^-{\mu_Y} \ar[rd]_-{\rho} &   & Y' \ar[ld] \\
  & W, &    \\
}
\]
such that $\mu_Y$ is a flip of $\rho$ or a divisorial contraction when $\mu_Y=\rho$, $\mu_X$ is a birational map obtained by a MMP over $W$ for some divisor on $X$, and $\pi'$ is $(K_X'+\Delta')$-Mori fiber space, where $\Delta'=(\mu_{X})_*\Delta$.
\end{lem}

\begin{rmk}\label{remarks for two-ray game}
In \cite{cerbo-svaldi}, they only deal with $K_Y$-flips and $K_Y$-divisorial contractions.
However, Lemma \ref{two-ray game} follows from the same argument as \cite[Lemma 3.3,  3.4]{cerbo-svaldi} except for $\pi'$-ampleness of $-(K_{X'}+\Delta')$.
In our case, since $-(K_X+\Delta)$ is big over $W$, $-(K_{X'}+\Delta')$ is also big over $W$.
Since the relative Picard rank of $X'$ over $Y'$ is equal to one, $-(K_{X'}+\Delta')$ is ample over $Y'$.

\end{rmk}

\begin{defn}\label{biratinal contraction on a tower of Mfs of cbf}
A birational map $\mu \colon (Y,\Gamma) \dashrightarrow (Y'\Gamma')$ of pairs is called a {\em birational contraction over towers of Mori fiber spaces} if
we have a following commutative diagram
\[
\xymatrix@C=15pt@R=20pt{
(Y,\Gamma) \ar@{-->}[rr]^-{\mu} \ar[d]_-{\pi_0} &  & (Y',\Gamma') \ar[d]^-{\pi_0'} \\
(Z_1,\Omega_1) \ar[d]_-{\pi_1} &  & (Z_1',\Omega_1') \ar[d]^-{\pi_1'} \\
\vdots \ar[d]_-{\pi_{r-1}} &  & \vdots \ar[d]^-{\pi_{r-1}'} \\
(Z_r,\Omega_r) \ar@{-->}[rr]^-{\mu_r} \ar[rd]_{\rho} &  & (Z_r',\Omega_r') \ar[ld] \\
& B, & \\
}
\]
where, 
\begin{itemize}
    \item $\pi_i$, $\pi_i'$ are Mori fiber spaces of canonical bundle formula type for all $0 \leq i \leq r-1$,
    \item $\rho$ is a $(K_{Z_r}+\Omega_r)$-negative extremal ray contraction which is birational,
    \item $\mu_r$ is a flip of $\rho$ or a divisorial contraction when $\mu_Y=\rho$,
    \item $\mu$ is a birational map which is a composition of steps of MMP over $B$, and
    \item $\Gamma'=\mu_*\Gamma$.
\end{itemize}
\end{defn}

\begin{rmk}
Divisorial contractions and flips obtained by a $(K_Y+\Gamma)$-negative extremal ray contraction are birational contractions over towers of Mori fiber spaces.
\end{rmk}

\begin{lem}\label{lift of the tower of Mori fiber spaces}
Let $\pi \colon (X,\Delta) \to (Y,\Gamma)$ be a Mori fiber space of canonical bundle formula type and $\mu_Y \colon (Y,\Gamma) \dashrightarrow (Y',\Gamma')$ is a birational contraction over towers of Mori fiber spaces.
Assume that $X$ has an int-amplified endomorphism $f$ such that $(X,\Delta)$ is an $f$-pair.
Then we obtain a following commutative diagram
\[
\xymatrix{
(X,\Delta) \ar@{-->}[r]^-{\mu_X} \ar[d]_-{\pi} &  (X',\Delta') \ar[d]^-{\pi'} \\
(Y,\Gamma) \ar@{-->}[r]_-{\mu_Y} &  (Y',\Gamma'),  \\
}
\]
where $\mu_X$ is a birational contraction over towers of Mori fiber spaces and $\pi'$ is a Mori fiber space of canonical bundle formula type.
Furthermore, $f^m$ induces an int-amplified endomorphism $f'$ of $X'$ such that $(X',\Delta')$ is an $f'$-pair.
\end{lem}

\begin{proof}
By the definition of birational contractions over towers of Mori fiber spaces,
we obtain the following diagram
\[
\xymatrix@C=15pt@R=20pt{
(Y,\Gamma) \ar@{-->}[rr]^-{\mu_Y} \ar[d]_-{\pi_0} &  & (Y',\Gamma') \ar[d]^-{\pi_0'} \\
(Z_1,\Omega_1) \ar[d]_-{\pi_1} &  & (Z_1',\Omega_1') \ar[d]^-{\pi_1'} \\
\vdots \ar[d]_-{\pi_{r-1}} &  & \vdots \ar[d]^-{\pi_{r-1}'} \\
(Z_r,\Omega_r) \ar@{-->}[rr]^-{\mu_r} \ar[rd]_{\rho} &  & (Z_r',\Omega_r') \ar[ld] \\
& B & \\
}
\]
as in Definition \ref{biratinal contraction on a tower of Mfs of cbf}.
Since we obtain the following sequence
\[
\xymatrix{
(X,\Delta) \ar[r]^-{\pi} & (Y,\Gamma) \ar[r]^-{\pi_0} & \cdots \ar[r]^-{\pi_{r-1}} &(Z_r,\Omega_r) \ar[r]^-{\rho} & B, 
}
\]
and this sequence satisfies the assumption of Corollary \ref{fano type of a tower of mfs},
$(X,\Delta)$ is of Fano type over $B$.
Since $\mu_Y$ is a composition of flips and divisorial contractions, we can take the first step of $\mu_Y$ as follows,
\[
\xymatrix{
Y \ar@{-->}[rr]^-{\mu_{Y_1}} \ar[rd]_-{\rho_1} &   & Y_1 \ar[ld] \\
& W,  &  \\
}
\]
where $\rho_1$ is an extremal ray contraction over $B$ and $\mu_{Y_1}$ is a flip of $\rho_1$ or divisorial contraction when $\rho_1=\mu_{Y_1}$.
Since $\rho_1$ is a morphism over $B$ and $(X,\Delta)$ is of Fano type over $B$, $(X,\Delta)$ is also of Fano type over $W$.
By Lemma \ref{two-ray game}, we obtain the following commutative diagram
\[
\xymatrix{
X \ar@{-->}[rr]^-{\mu_{X_1}} \ar[d]_-{\pi} &   & X_1 \ar[d]^-{\pi_{X_1}} \\
Y \ar@{-->}[rr]^-{\mu_{Y_1}} \ar[rd]_-{\rho} &   & Y_1 \ar[ld] \\
  & W &    \\
}
\]
as in Lemma \ref{two-ray game}.
Let $\Delta_1=(\mu_{X_1})_*\Delta$, then $(X_1,\Delta_1)$ is of Fano type over $B$ by Proposition \ref{fano type birational contraction}.
Repeating the construction, we obtain the following commutative diagram
\[
\xymatrix{
(X,\Delta) \ar@{-->}[r]^-{\mu_X} \ar[d]_-{\pi} &  (X',\Delta') \ar[d]^-{\pi'} \\
(Y,\Gamma) \ar@{-->}[r]_-{\mu_Y} &  (Y',\Gamma'),  \\
}
\]
such that $\mu_X$ is a birational map which is a composition of steps of MMP over $B$ and $(X',\Delta')$ is of Fano type over $Y'$ by the construction, where $\Delta'=(\mu_X)_*\Delta$.
In particular $f^m$ induces an int-amplified endomorphism $f'$ of $X'$ for some $m$.

Finally, we prove that $(X',\Delta')$ is an $f'$-pair and $\pi' \colon (X',\Delta') \to (Y,\Gamma)$ is a Mori fiber space of canonical bundle formula type.
First, since we have
\[
R_{\Delta'}=(\mu_X)_*R_{\Delta,m}=(\mu_X)_*(R_{f^m}+\Delta-(f^m)^*\Delta) \geq 0,
\]
$(X',\Delta')$ is an $f'$-pair.
Next, we take a prime divisor $E'$ on $Y'$ and $\pi'^*E'=m_{E'}F'$ for some $m_{E'}$ and prime divisor $F'$. 
Since $\mu_X$ and $\mu_Y$ are birational contractions, we can consider the strict transform $E$ and $F$ of $E'$ and $F'$, respectively.
Since $\mu_X$ and $\mu_Y$ are isomorphism on the generic points of $F$ and $E$, respectively, we have $\pi^*E=m_{E'}F$, $\ord_F(\Delta)=\ord_{F'}(\Delta')$ and $\ord_E(\Gamma)=\ord_{E'}(\Gamma')$.
Hence $\pi'$ is a Mori fiber space of canonical bundle formula type, since so is $\pi$.
\end{proof}

\begin{defn}\label{definition of steps of MMP of canonical bundle formula type}
A sequence of biratinal maps and morphisms of pairs
\[
\xymatrix@R=15pt{
(X,\Delta) \ar@{-->}[r]^-{\sigma_0} & (X',\Delta') \ar[d]^-{\pi_0} & & &  &    \\
& (X_1,\Delta_1) \ar@{-->}[r]^-{\sigma_1} & (X'_1,\Delta'_1) \ar[d]^-{\pi_1} & &  &   \\
& & (X_2,\Delta_2) \ar@{-->}[r]^-{\sigma_1} & \ar@{}[dr]|\ddots &  & \\
&  &  &      & \ar@{-->}[r]^-{\sigma_r}  &  (X_r',\Delta_r') \ar[d]^-{\pi_r} \\
&  &   &   &   & (W,\Delta_W), \\
}
\]
is called {\em sequence of steps of MMP of canonical bundle formula type} if 
\begin{itemize}
    \item $X$ is a normal $\Q$-factorial projective variety, $\Delta$ is an effective $\Q$-Weil divisor such that $(X,\Delta)$ is klt,
    \item $X_i \dashrightarrow X_i' \to X_{i+1}$ is obtained by $(K_{X_i}+\Delta_i)$-MMP for all $0 \leq i \leq r$, and
    \item $\pi_i$ is a Mori fiber space of canonical bundle formula type for all $0 \leq i \leq r$, where $\Delta_i'=(\sigma_i)_*\Delta_i$.
\end{itemize}

Furthermore, the above sequence is called {\em maximal} if $K_W+\Delta_W$ is pseudo-effective.
\end{defn}

\begin{rmk}\label{steps of MMP of canonical bundle formula type preserves f-pair}
If $K_W+\Delta_W$ is not pseudo-effective, then we can run MMP for $(W,\Delta_W)$ and we obtain a Mori fiber space.
In particular, there exists a maximal sequence of steps of MMP of canonical bundle formula type.

If $X$ has an int-amplified endomorphism, then for every surjective endomorphism $f$, the above sequence is $f^m$-equivarient for some positive integer $m$ by Theorem \ref{equivariant extremal ray contraction}.
We denote the induced endomorphisms on $X_i, X_i'$ and $W$ by $f_i, f_i'$ and $g$, respectively.
If $(X,\Delta)$ is an $f$-pair, then $(X_i,\Delta_i), (X_i',\Delta_i')$ and $(W,\Delta_W)$ are an $f_i$-pair, an $f_i'$-pair, a $g$-pair, respectively by Proposition \ref{mfs of cbf type for f-pairs} and the proof of Lemma \ref{lift of the tower of Mori fiber spaces}.
In particular, $-(K_W+\Delta_W)$ is pseudo-effective by Lemma \ref{first properties of int-ampl endo}.
If we further assume that the sequence is maximal, then $K_W+\Delta_W \sim_{\Q} 0$.
\end{rmk}

\begin{thm}\label{construction of a towers of mfs}
Let $(X,\Delta) \dashrightarrow \cdots \dashrightarrow (W,\Delta_W)$ be a sequence of steps of MMP of canonical bundle formula type and $f$ be an int-amplified endomorphism of $X$ such that $(X,\Delta)$ is an $f$-pair.
Then we obtain a sequence of birational contractionas over towers of Mori fiber spaces
\[
\mu \colon (X,\Delta) \dashrightarrow (X',\Delta')
\]
and a sequence of Mori fiber spaces of canonical bundle formula type
\[
(X',\Delta') \to (X_1',\Delta_1') \to \cdots \to (X_r',\Delta_r') \to (W,\Delta_W).
\]
Furthermore, $(X',\Delta')$ is of Fano type over $W$.
\end{thm}

\begin{proof}
We prove Theorem \ref{construction of a towers of mfs} by the induction on the length of the sequences of steps of MMP of canonical bundle formula type.
We take the first step $\pi \colon (X,\Delta) \dashrightarrow (Y,\Gamma)$ in the above sequence, in particular, $\pi$ is a flip, a divisorial contraction or a Mori fiber space obtained by a $(K_X+\Delta)$-negative extremal ray contraction. By the induction hypothesis, we obtain a sequence of birational contractionas over towers of Mori fiber spaces
\[
\mu_Y \colon (Y,\Gamma) \dashrightarrow (Y',\Gamma')
\]
and a sequence of Mori fiber spaces of canonical bundle formula type
\[
(Y',\Gamma') \to (Y_1',\Gamma_1') \to \cdots \to (Y_{r'-1},\Gamma_{r-1}') \to (W,\Delta_W).
\]
If $\pi$ is a birational map, $\mu_Y \circ \pi$ is also a composition over towers of canonical bundle formula type.
Hence we may assume that $\pi \colon (X,\Delta) \to (Y,\Gamma)$ is a Mori fiber space of canonical bundle formula type.
By repeatedly applying Lemma \ref{lift of the tower of Mori fiber spaces}, we obtain the following commutative diagram
\[
\xymatrix{
(X,\Delta) \ar@{-->}[r]^-{\mu} \ar[d]_-{\pi} &  (X',\Delta') \ar[d]^-{\pi'} \\
(Y,\Gamma) \ar@{-->}[r]_-{\mu_Y} &  (Y',\Gamma'),  \\
}
\]
where $\mu$ is a sequence of birational contractionas over towers of canonical bundle formula type and $\pi'$ is a Mori fiber space of canonical bundle formula type.
Hence, we obtain the following composition of Mori fiber spaces of canonical bundle formula type
\[
(X',\Delta') \to (Y',\Gamma') \to (Y_1',\Gamma_1') \to \cdots \to (Y_{r'-1},\Gamma_{r-1}') \to (W,\Delta_W).
\]
By Lemma \ref{equivariant extremal ray contraction}, $\mu$ is $f^m$-equivariant for some $m$, so $f^m$ induces an int-amplified endomorphism $f'$ of $X'$ such that $(X',\Delta')$ is an $f'$-pair.
By Lemma \ref{fano type of a tower of mfs}, we obtain $(X',\Delta')$ is of Fano type over $W$.
\end{proof}

\section{Maximal sequence of steps of MMP}

In this section, we construct $\widetilde{X}$ and $A$ in Theorem \ref{main thm} by studying maximal sequences of steps of MMP of canonical bundle formula type.
First, we construct $A$ as a covering of the output of a sequence of steps of MMP.
Lifting this covering, we construct $\widetilde{X}$.

\begin{lem}\label{totally invariant prime divisors}
Let $f \colon X \to X$ ba a surjective endomorphism of a normal variety $X$
and $\Sigma$ a finite set of the prime divisors on $X$.
Assume that for any $E \in \Sigma$ and an irreducible component $F$ of $f^{-1}(E)$, we have $F \in \Sigma$.
Then for any $E \in \Sigma$, $f^{-m}(E)=E$ as sets for some positive integer $m$.
\end{lem}

\begin{proof}
Let $\Sigma = \{ E_1, \ldots , E_N \}$.
Since $f^{-1}(E_i)$ and $f^{-1}(E_j)$ have no common irreducible component, we have $f^{-1}(E_i)$ is irreducible for all $i$.
It means that $f^{-1}$ induces the one-to-one corresponding of $\Sigma$.
In particular, for some positive integer $m$, $f^{-m}$ is identity.
\end{proof}

\begin{defn}
A morphism $\mu \colon (\widetilde{X},\widetilde{\Delta}) \to (X, \Delta)$ of pairs is called a quasi-\'etale cover of $(X,\Delta)$ if
\begin{itemize}
    \item $\mu \colon \widetilde{X} \to X$ is a finite surjective morphism of normal projective varieties, 
    \item $\Delta$ is an effective $\Q$-Weil divisor on $X$, and
    \item $\widetilde{\Delta}=\mu^*\Delta - R_\mu$ is effective.
\end{itemize}
\end{defn}

\begin{rmk} Let $\mu \colon (\widetilde{X},\widetilde{\Delta}) \to (X, \Delta)$ be a quasi-\'etale cover of $(X,\Delta)$ and $K_X+\Delta$ is $\Q$-Cartier. Then
\begin{itemize}
    \item $K_{\widetilde{X}}+\widetilde{\Delta}= \mu^*(K_X+\Delta)$, and in particular,  $K_{\widetilde{X}}+\widetilde{\Delta}$ is $\Q$-Cartier, 
    \item $(X,\Delta)$ is klt if and only if $(\widetilde{X},\widetilde{\Delta})$ is klt by \cite[Proposition 5.20]{kollar-mori}, and
    \item if $\Delta=0$, then $\mu$ is \'etale in codimension one.
\end{itemize}
\end{rmk}

\begin{prop}\label{quasi-rtale preserves f-pairness}
We consider the following commutative diagram
\[
\xymatrix{
(\widetilde{X},\widetilde{\Delta}) \ar[r]^\mu \ar[d]_{\widetilde{f}} & (X,\Delta) \ar[d]^f  \\
(\widetilde{X},\widetilde{\Delta}) \ar[r]_\mu & (X, \Delta),
}
\]
where $\mu$ is a quasi-\'etale cover of $(X,\Delta)$, $f$ and $\widetilde{f}$ are surjective endomorphisms.
Then we have
\[
R_{\widetilde{\Delta}}=R_{\widetilde{f}}+\widetilde{\Delta}-\widetilde{f}^*\widetilde{\Delta}=\mu^*(R_f+\Delta-f^*\Delta)=\mu^*R_\Delta,
\]
and in particular, $(X,\Delta)$ is an $f$-pair if and only if $(\widetilde{X},\widetilde{\Delta})$ is an $\widetilde{f}$-pair.
\end{prop}

\begin{proof}
First, since $f\circ \mu = \mu \circ \widetilde{f}$, we have
\[
R_\mu+\mu^*R_f=R_{\widetilde{f}}+\widetilde{f}^*R_\mu.
\]
Hence we have
\begin{eqnarray*}
R_{\widetilde{\Delta}}&=&R_{\widetilde{f}}+\widetilde{\Delta}-\widetilde{f}^*\widetilde{\Delta} \\
&=& R_{\widetilde{f}}+\mu^*\Delta-R_\mu- \widetilde{f}^*(\mu^*\Delta-R_\mu) \\
&=& \mu^*(R_f+\Delta-f^*\Delta)=\mu^*R_\Delta.
\end{eqnarray*}
\end{proof}

\begin{eg}\label{example; quasi-etale cover}
In Example \ref{example; f-pair}, $\mu \colon (E,0) \to (\P^1,\Delta)$ is a quasi-\'etale cover of $(\P^1,\Delta)$.
In particular. $(\P^1,\Delta)$ is an $h$-pair by Proposition \ref{quasi-rtale preserves f-pairness}.
Furthermore, in this case, $K_{\P^1}+\Delta$ is $\Q$-linearly trivial, so this is also an example of the following proposition.
\end{eg}

\begin{prop}[cf. Example \ref{example; quasi-etale cover}]\label{convering theorem}
Let $X$ be a normal $\Q$-factirial projective variety admitting an int-amplified endomorphism $f$ and $\Delta$ an effective $\Q$-Weil divisor on $X$ such that $(X,\Delta)$ is a klt $f$-pair.
Assume that $K_X+\Delta \sim_{\Q} 0$.
Then there exits a quasi-\'etale cover $\mu \colon (A,0) \to (X,\Delta) $ of $(X,\Delta)$ by an abelian variety $A$.
Furthermore, if $\phi \colon X \to X$ is a surjective endomorphism of $X$ such that $(X,\Delta)$ is a $\phi$-pair, then $\phi$ lifts to $A$.

\end{prop}

\begin{proof}
By the proof of \cite[Theorem 1.6]{yoshikawa19}, it is enough to show that $\Delta$ has standard coefficients, that is, for any prime divisor $E$ on $X$, $\Delta$ satisfies $\ord_E(\Delta)=\frac{m-1}{m}$ for some positive integer $m$.
Let $\Sigma$ be the set of all prime divisors $E$ on $X$ such that $\ord_E(\Delta)$ is not standard. Since $0$ is standard, $\Sigma$ is a finite set.
We take $E \in \Sigma$ and an irreducible component $F$ of $f^{-1}(E)$.

Suppose that $F \notin \Sigma$, then $\ord_F(\Delta)=\frac{m-1}{m}$ for some positive integer $m$.
We note that
\[
0 \leq R_{\Delta}=R_f +\Delta-f^*\Delta \sim K_X+\Delta-f^*(K_X+\Delta)
\sim_{\Q} 0,
\]
hence, we have $R_{\Delta}=0$.
In particular, we have
\[
0=\ord_F(R_f+\Delta-f^*\Delta)=r-1 + \frac{m-1}{m} - r \cdot \ord_E(\Delta),
\]
where $r=\ord_F(f^*(E))$.
It means that
\[
\ord_E(\Delta)=\frac{rm-1}{rm},
\]
it is contradiction to $E \in \Sigma$.

Hence $F$ is contained in $\Sigma$ and $\Sigma$ satisfies the assumption of Lemma \ref{totally invariant prime divisors}.
Next suppose that there exists an element $E \in \Sigma$.
Then we have $f^{-l}(E)=E$ as sets for some positive integer $l$.
Therefore, we have
\[
0=\ord_E(R_{f^l}+\Delta-(f^l)^*\Delta)=r'-1+\ord_E(\Delta)-r'\ord_E(\Delta),
\]
where $(f^l)^*E=r'E$.
Since $f^l$ is int-amplified, $r'$ is larger than one by \cite[Theorem 3.3]{meng17}, so we have $\ord_E(\Delta)=1$ .
However, it is contradiction to the fact that $(X,\Delta)$ is klt.
In conclusion, we obtain $\Sigma = \emptyset$ and $\Delta$ has standard coefficients.

\end{proof}

\begin{thm}[\textup{cf. \cite[Lemma 4.12]{matsuzawa-yoshikawa19}}]\label{lifting of cover of abelian variety}
Let $(X,\Delta) \dashrightarrow \cdots \dashrightarrow (W,\Delta_W)$ be a maximal sequence of steps of MMP of canonical bundle formula type and $f$ be an int-amplified endomorphism of $X$ such that $(X,\Delta)$ is an $f$-pair.
Then $\pi_X \colon X \dashrightarrow W$ is a morphism and we obtain the following commutative diagram
\[
\xymatrix{
(\widetilde{X},\widetilde{\Delta}) \ar[r]^-{\widetilde{\pi}} \ar[d]_-{\mu_X} & (A,0) \ar[d]^-{\mu_W} \\
(X,\Delta) \ar[r]_-{\pi_X} & (W,\Delta_W), \\
}
\]
where
\begin{itemize}
    \item $A$ is an abelian variety,
    \item $\mu_X, \mu_W$ are quasi-\'etale covers of $(X,\Delta), (W,\Delta_W)$, respectively,
    \item $\widetilde{X}$ is the normalization of the main component of $X \times_W A$, and
    \item $R_{\Delta}$ has no vertical components of $\pi_X$.
\end{itemize}
Furthermore, if $\phi$ is a surjective endomorphism of $X$ such that $(X,\Delta)$ is a $\phi$-pair, then there exists the following commutative diagram
\[
\xymatrix{
X  \ar[d]_{\phi^m}& \widetilde{X} \ar[l]_{\mu_X} \ar[r]^{\widetilde{\pi}} \ar[d]^{\widetilde{\phi}} & A \ar[d]^{\phi_A} \\
X  & \widetilde{X} \ar[l]^{\mu_X} \ar[r]_{\widetilde{\pi}} & A,
}
\]
for some positive integer $m$, where $\widetilde{\phi}$ and $\phi_A$ are surjective endomorphism.
\end{thm}

\begin{proof}
First, we prove the last assertion.
Let $\phi$ be a surjective endomorphism of $X$ such that $(X,\Delta)$ is a $\phi$-pair.
By Remark \ref{steps of MMP of canonical bundle formula type preserves f-pair}, $\phi^m$ induces the surjective endomorphism $\psi$ of $W$ such that $(W,\Delta_W)$ is a $\psi$-pair.
By Proposition \ref{convering theorem}, $\psi$ lifts to $A$ denoted by $\widetilde{\psi}$.
Since $\widetilde{X}$ is the normalization of the main component of $X \times_W A$, $\widetilde{X}$ has an endomorphism $\widetilde{\phi}$ induced by $\phi^m, \psi$ and $\widetilde{\psi}$.
This satisfies the conditions in Theorem \ref{lifting of cover of abelian variety}.

Next, we prove Theorem \ref{lifting of cover of abelian variety} by the induction on the length of the maximal sequences of steps of MMP of canonical bundle formula type.
We note that $(K_W+\Delta_W) \sim_{\Q} 0$ and $(W,\Delta_W)$ is a $g$-pair for some int-amplified endomoprhism $g$ on $W$ by Remark \ref{steps of MMP of canonical bundle formula type preserves f-pair}. 
Hence by Proposition \ref{convering theorem}, we can construct a quasi-\'etale cover of $(W,\Delta_W)$ by an abelian variety $A$.
This is the proof of the case where the length is equal to $0$.

Let $\rho \colon (X,\Delta) \dashrightarrow (Y,\Gamma)$ be the first step of the above sequence.
We may assume that $\rho$ is $f$-equivariant replacing $f$ by some iterate.
By the induction hypothesis, we obtain the following commutative diagram
\[
\xymatrix{
(\widetilde{Y},\widetilde{\Gamma}) \ar[r]^-{\widetilde{\pi}_Y} \ar[d]_-{\mu_Y} & (A,0) \ar[d]^-{\mu_W} \\
(Y,\Gamma) \ar[r]_-{\pi_Y} & (W,\Delta_W), \\
}
\]
where
\begin{itemize}
    \item $A$ is an abelian variety,
    \item $\mu_Y, \mu_W$ are quasi-\'etale covers of $(Y,\Gamma), (W,\Delta_W)$, respectively,
    \item $\widetilde{Y}$ is the normalization of the main component of $Y \times_W A$, and
    \item $R_{\Gamma}$ has no vertical components of $\pi_Y$.
\end{itemize}

First, we consider the case where $\rho$ is a flip.
Let $\widetilde{X}$ be the normalization of $X$ in $K(\widetilde{Y})$. 
Since $\widetilde{X} \dashrightarrow \widetilde{Y}$ is isomorphism in codimension one,

\[
\mu_X \colon (\widetilde{X}, \widetilde{\Delta}) \to (X,\Delta)
\]
is a quasi-\'etale cover of $(X,\Delta)$.
In particular, $(\widetilde{X}, \widetilde{\Delta})$ is klt,
so $\widetilde{X}$ has at worst rational singularities by \cite[Theorem 5.22]{kollar-mori}.
By \cite[Lemma 8.1]{kawamata85}, 
\[
\widetilde{\pi}_X \colon \widetilde{X} \dashrightarrow A
\]
is a morphism.
By \cite[Lemma 4.13]{matsuzawa-yoshikawa19}, $\pi_X$ is also a morphism.
The other statements follow from the smallness of $\rho$ and the proof of \cite[Lemma 4.12]{matsuzawa-yoshikawa19}.

Next, we consider the case where $\rho$ is a divisorial contraction or a Mori fiber space.
We obtain the following commutative diagram
\[
\xymatrix{
\widetilde{X} \ar[r]^{\widetilde{\rho}} \ar[d]_{\mu_X} & \widetilde{Y} \ar[r]^{\widetilde{\pi}_Y} \ar[d]_{\mu_Y} & A \ar[d]^{\mu_W}  \\
X \ar[r]_{\rho} & Y \ar[r]_{\pi_Y} & W, \\
}
\]
where $\widetilde{X}$ is the normalization of the main component of $X \times_W A$.
\begin{claim}\label{ramification divisor has no vertical component}
$R_{\Delta}$ has no vertical components.
\end{claim}
\begin{proof}
If $\rho$ is a divisorial contraction, then there exists the unique exceptional prime divisor $E$ on $X$.
It is enough to show that $\pi_X(E)=W$.
It follows from the fact that $\pi_X(E)$ is totally invarinat (cf. \cite[Lemma 4.12]{matsuzawa-yoshikawa19}).

If $\rho$ is a Mori fiber space, then
\[
\rho \colon (X,\Delta) \to (Y,\Gamma)
\]
is a Mori fiber space of canonical bundle formula type since $\rho$ is in steps of MMP of canonical bundle formula type.
By Proposition \ref{mfs of cbf type for f-pairs}, $R_\Delta-\rho^*R_{\Gamma}$ has no vertical component of $\rho$.
Since $R_\Gamma$ has no vertical components of $\pi_Y$, $R_\Delta$ has no vertical component of $\pi_X$.
\end{proof}

Next, we prove that $\widetilde{\Delta}=(\mu_X)^*\Delta-R_{\mu_X}$ is effective, that is, $\mu_X$ is a quasi-\'etale cover of $(X,\Delta)$.
Let $\Sigma$ be the set of all prime divisors $E$ on $\widetilde{X}$ with $\ord_E(\widetilde{\Delta}) <0$.
We may assume that $f$ lifts to $\widetilde{X}$ denoted by $\widetilde{f}$ replacing $f$ by some iterate.
Suppose that $\Sigma \neq \emptyset$, and take $E \in \Sigma$ and an irreducible component $F$ of $f^{-1}(E)$.
Then $\ord_E(R_{\mu_X}) >0$.
By the construction of $\widetilde{X}$, $R_{\mu_X}$ has only vertical components of $\widetilde{\pi}_X$ (cf. \cite[Lemma 4.12]{matsuzawa-yoshikawa19}).
Hence $E$ is a vertical component, so $F$ is alzo a vertical component.
Let
\[
R_{\widetilde{\Delta}}=R_{\widetilde{f}}+\widetilde{\Delta}-\widetilde{f}^*\widetilde{\Delta}.
\]
By Proposition \ref{quasi-rtale preserves f-pairness}, we obtain $(\mu_X)^*R_\Delta=R_{\widetilde{\Delta}}$,
in particular, it is effective and has no vertical component of $\widetilde{\pi}_X$.
Hence, $\ord_F(R_{\widetilde{\Delta}})=0$ and
\[
\ord_F(\widetilde{\Delta}) \leq \ord_F(R_{\widetilde{f}}+\widetilde{\Delta}) = \ord_F(\widetilde{f}^*\widetilde{\Delta}) =r \cdot \ord_E(\widetilde{\Delta}) < 0,
\]
where $r=\ord_F(\widetilde{f}^*(E))$.
By Lemma \ref{totally invariant prime divisors}, we may assume that $\widetilde{f}^*(E)=r'E$ for some $r'>1$ replacing $f$ by some iterate.
Then we have
\[
\ord_E(\widetilde{\Delta}) \leq r'\ord_E(\widetilde{\Delta}),
\]
but it is contradiction to the fact that $\ord_E(\widetilde{\Delta})<0$.
It means that $\Sigma = \emptyset$ and $\widetilde{\Delta}$ is effective.
\end{proof}

\section{Proof of the main theorem}
In this section, we prove the main theorem.

\begin{prop}[\textup{cf. \cite[Lemma 1.1]{fujino-gongyo}}]\label{quasi-etale preserves fano ytpe}
Let $\mu \colon (Y,\Gamma) \to (X,\Delta)$ be a quasi-\'etale cover of $(X,\Delta)$ and $b \colon X \to B$ a surjective morphism of projective varieties.
Then $(X,\Delta)$ is of Fano type over $B$ if and only if
 $(Y,\Gamma)$ is of Fano type over $B$.
\end{prop}

\begin{proof}
First, we assume that $(X,\Delta)$ is of Fano type over $B$.
Since $-(K_X+\Delta)$ is big over $B$, then
\[
-(K_Y+\Gamma)=-\mu^*(K_X+\Delta)
\]
is big over $B$.
By Proposition \ref{first property of Fano type}, there exists an effective $\Q$-Weil divisor $D$ on $X$ such that $(X,\Delta+D)$ is klt and $K_X+\Delta+D \equiv_{B} 0$.
Then $K_Y+\Gamma+\mu^*D \equiv_{B} 0$ and $(Y,\Gamma+\mu^*D)$ is klt.
By Proposition \ref{first property of Fano type}, $(Y,\Gamma)$ is of Fano type over $B$.

Next, we assume that $(Y,\Gamma)$ is of Fano type over $B$.
Then by the same argument as above, $-(K_X+\Delta)$ is big over $B$ and
we can take an effective $\Q$-Weil divisor $D_Y$ on $Y$ such that $K_Y+\Gamma+D_Y \equiv_B0$ and $(Y,\Gamma+D_Y)$ is klt.
Let
\[
D_X=\deg(\mu)^{-1}\mu_*(R_\mu+\Gamma+D_Y)-\Delta.
\]
Since $\Gamma=\mu^*\Delta-R_\mu$, we have
\[
D_X=\deg(\mu)^{-1}\mu_*D_Y.
\]
By the proof of \cite[Lemma 1.1]{fujino-gongyo}, $K_X+\Delta+D_X \equiv_B0$ and $(X,\Delta+D_X)$ is klt, and in particular, $(X,\Delta)$ is of Fano type over $B$. 
\end{proof}

\begin{prop}\label{equivariant birational map preserves fano type}
We consider the following commutative diagram
\[
\xymatrix{
f \acts X \ar@{-->}[rr]^{\pi} \ar[rd]_-{b_X} & & Y \racts g \ar[ld]^{b_Y} \\
& h \acts B,  &  \\
}
\]
where
\begin{itemize}
    \item $X$, $Y$, $B$ are normal projective varieties,
    \item $f$, $g$, $h$ are int-amplified endomorphisms, and
    \item $\pi$ is a biratinal map.
\end{itemize}
Let $\Delta$ be an effective $\Q$-Weil divisors on $X$ such that $(X,\Delta)$ is an $f$-pair.
Assume that $(Y,\Gamma)$ is of Fano type over $B$ and $(Y,\Gamma)$ is a $g$-pair, where $\Gamma=\pi_*\Delta$.
Then $(X,\Delta)$ is also of Fano type over $B$.
\end{prop}

\begin{proof}
We take $W$ the normalization of the graph of $\pi$ as follows,
\[
\xymatrix{
& \ar[ld]_-{\mu_X} W \ar[rd]^-{\mu_Y} & \\
X \ar@{-->}[rr]_-\pi & & Y.
}
\]
Then $f$ and $g$ induce the endomorphism $f'$ of $W$ such that $(W,\Omega=(\mu_X)^{-1}_*\Delta)$ is an $f'$-pair after iterating $f'$.
Indeed, since $(\mu_X)_*R_\Omega=R_\Delta$ is effective, it is enough to show that $\ord_E(R_\Omega) \geq 0$ for every exceptional prime divisor $E$.
Since we have
\[
R_\Omega=R_{f'}+\Omega-(f')^*\Omega
\]
and $\ord_E((f')^*\Omega)=0$, we obtain $\ord_E(R_\Omega) \geq 0$.
By Proposition \ref{fano type birational contraction}, it is enough to show that $(W,\Omega)$ is of Fano type over $B$.
Hence, we may assume that $\pi$ is a morphism.

By Proposition \ref{first property of Fano type}, there exits an effective $b_Y$-big $\Q$-Weil divisor $D_Y$ on $Y$ such that $K_Y+\Gamma+D_Y\equiv_{B} 0$ and $(Y,\Gamma+D_Y)$ is klt. 
We define a $\Q$-Weil divisor $D_X$ on $X$ as 
\[
K_X+\Delta+D_X=\pi^*(K_Y+\Gamma+D_Y).
\]
Then $K_X+\Delta+D_X\equiv_B 0$.
Next, let
\[
D_{Y.i}=\deg(g^i)^{-1}(g_i)_*(R_{g^i}+\Gamma+D_Y)-\Gamma
\]
and
\[
D_{X.i}=\deg(f^i)^{-1}(f_i)_*(R_{f^i}+\Delta+D_X)-\Delta.
\]
Then by the same argument as in the proof of Lemma \ref{Fano type if the ramification divisor contains ample div} and \cite[Lemma 1.1]{fujino-gongyo},
$D_{Y,i}$ is effective, $(Y,\Gamma+D_{Y,i})$ is klt, $K_Y+\Gamma+D_{Y,i}\equiv_B0$ and
\[
K_X+\Delta+D_{X,i}=\pi^*(K_Y+\Gamma+D_{Y,i}).
\]

\begin{claim}\label{equivariant birational map preserves fano type claim1}
There exists a positive integer $m$ such that $D_{X,m}$ is effective.
\end{claim}

\begin{claimproof}
Since $\pi_*D_{X,i}=D_{Y,i}$, it is enough to show that for any exceptional prime divisor $E$ on $X$ such that $\ord_E(D_{X,m})$ is effective for some $m$.
We fix an exceptional prime divisor $E$ on $X$.
We may assume that $E$ is totally invariant under $f$ replacing $f$ by some iterate.
Let $f^*E=rE$ for some $r>1$ since $f$ is int-amplified.
Let $f_*E=eE$, then $re=\deg(f)$.
Let 
\[
a = a_E(Y,\Gamma+D_Y)>0,
\]
where $a_E(Y,\Gamma+D_Y)$ is the log discrepancy of $(Y,\Gamma+D_Y)$ with respect to $E$.
Then $\ord_E(\Delta+D_Y)=1-a$ by the definition of the log discrepancy.
Then we have
\begin{eqnarray*}
\ord_E(D_{X,i})&=& \ord_E(\deg(f^i)^{-1}(f_i)_*(R_{f^i}+\Delta+D_X)-\Delta) \\
&=& \deg(f^i)^{-1}e^i(r^i-1+\ord_E(\Delta+D_X))-\ord_E(\Delta) \\
&=& \frac{r^i-a}{r^i}-\ord_E(\Delta) \\
&=& 1-\ord_E(\Delta)-\frac{a}{r^i} \geq 0
\end{eqnarray*}
for enough large $i$.
\end{claimproof}

By the claim, we obtain an effective $\Q$-Weil divisor $D_{X,m}$ such that $K_X+\Delta+D_{X,m}\equiv_B0$ and $(X,\Delta+D_{X,m})$ is klt.
It is enough to show that $D_{X,m}$ is big over $B$.
Since we have
\[
R_{\Delta,l}\sim K_X+\Delta-(f^l)^*(K_X+\Delta)\equiv_B (f^l)^*D_{X,m}-D_{X,m},
\]
it is enough to show that $R_{\Delta,l}+D_{X,m}$ is big over $B$ for some $l$.
Since we have
\[
R_{\Gamma} \sim K_Y+\Gamma-g^*(K_Y+\Gamma) \equiv_B g^*D_{Y,m}-D_{Y,m}
\]
and $D_{Y,m}\equiv_B -(K_Y+\Gamma)$ is big over $B$,
$R_{\Gamma}+D_{Y,m}$ is big over $B$.
Hence it is enogh to show that the following claim.
\begin{claim}
$R_{\Delta,l}+D_{X,m} \geq \pi^*(R_{\Gamma}+D_{Y,m})$ for some positive integer $l$.
\end{claim}

\begin{claimproof}
Since we have
\[
\pi_*(R_{\Delta,l}+D_{X,m}) \geq \pi_*(R_{\Delta}+D_{X,m}) = R_{\Gamma}+D_{Y,m},
\]
it is enough to consider the coefficients with respect to the exceptional prime devisors.
We take an exceptional prime divisor $E$ on $X$ and we may assume that $E$ is totally invariant under $f$.
Let $f^*E=rE$.
Then we have
\begin{eqnarray*}
\ord_E(R_{\Delta,l}+D_{X,m}) 
&\geq& \ord_E(R_{\Delta,l}) =\ord_E(R_{f^l}+\Delta-(f^l)^*\Delta) \\
&=& r^l-1+\ord_E(\Delta)-r^l\ord_E(\Delta) \\
&=& (r^l-1)(1-\ord_E(\Delta)) \geq \ord_E(\pi^*(R_\Gamma+D_{Y,m}))
\end{eqnarray*}
for large enough $l$,
since $\ord_E(\pi^*(R_\Gamma+D_{Y,m}))$ does not depend on $l$.
\end{claimproof}
\end{proof}

\begin{thm}\label{main thm for pairs}
Let $X$ be a normal $\Q$-factorial projective variety admitting an int-amplified endomorphism $f$ and $\Delta$ an effective $\Q$-Weil divisor on $X$ such that $(X,\Delta)$ is a klt $f$-pair.
Then there exist following morphisms and pairs
\[
\xymatrix{
(X,\Delta) & (\widetilde{X},\widetilde{\Delta}) \ar[r]^-{\widetilde{\pi}} \ar[l]_-{\mu} & A, 
 \\
}
\]
such that
\begin{itemize}
    \item $\mu$ is a quasi-\'etale cover of $(X,\Delta)$,
    \item $\widetilde{\pi}$ is a fiber space,
    \item $A$ is an abelian variety, 
    \item $\widetilde{X}$ is a normal variety, and
    \item $(\widetilde{X},\widetilde{\Delta})$ is of Fano type over $A$.
\end{itemize}
In particular, $\widetilde{\pi}$ is the albanese morphism.
Moreover, if $\phi$ is a surjective endomorphism of $X$ such that $(X,\Delta)$ is a $\phi$-pair, then we obtain the following diagram
\[
\xymatrix{
X  \ar[d]_{\phi^m}& \widetilde{X} \ar[l]_\mu \ar[r]^{\widetilde{\pi}} \ar[d]_{\widetilde{\phi}} & A \ar[d]_{\phi_A} \\
X & \widetilde{X} \ar[l]^\mu \ar[r]_{\widetilde{\pi}} & A \\
}
\]
for some $m$, where $\widetilde{\phi}$ and $\phi_A$ are endomorphisms.
\end{thm}

\begin{proof}
Let $(X,\Delta) \dashrightarrow \cdots \dashrightarrow (W,\Delta_W)$ be a maximal sequence of steps of MMP of canonical bundle formula type.
By Theorem \ref{lifting of cover of abelian variety}, we obtain the following commutative diagram
\[
\xymatrix{
(\widetilde{X},\widetilde{\Delta}) \ar[r]^-{\widetilde{\pi}} \ar[d]_-{\mu} & (A,0) \ar[d]^-{\mu_W} \\
(X,\Delta) \ar[r]_-{\pi_X} & (W,\Delta_W) \\
}
\]
as in Theorem \ref{lifting of cover of abelian variety}. 

We prove that $(\widetilde{X},\widetilde{\Delta})$ is of Fano type over $A$.
By Proposition \ref{quasi-etale preserves fano ytpe}, it is enough to show that $(X,\Delta)$ is of Fano type over $W$.
By Theorem \ref{construction of a towers of mfs}, we obtain a sequence of birational contractionas over towers of Mori fiber spaces $(X,\Delta) \dashrightarrow (X',\Delta')$ such that $(X',\Delta')$ is of Fano type over $W$.
By the construction, replacing $f$ by some iterate, we obtain the following commutative diagram
\[
\xymatrix{
f \acts X \ar@{-->}[rr] \ar[rd] & & X' \racts f' \ar[ld] \\
& h \acts W.  &  \\
}
\]
By Proposition \ref{equivariant birational map preserves fano type}, $(X,\Delta)$ is of Fano type over $W$.
Moreover, a general fiber of $\widetilde{\pi}$ is of Fano type, in particular, it is rationally connected by \cite{hacon-mackernan07}.
Since $\widetilde{\pi}$ is a fiber space, $\widetilde{\pi}$ is the albanese morphism.
\end{proof}

\begin{thm}[Theorem \ref{main thm}, Theorem \ref{main thm; lifting}]\label{main thm'}
Let $X$ be a normal $\Q$-factorial klt projective variety admitting an int-amplified endomorphism.
Then there exists a quasi-\'etale finite cover $\mu \colon \widetilde{X} \to X$ such that
the albanese morphism $\mathrm{alb}_{\widetilde{X}} \colon \widetilde{X} \to A$ is a fiber space and $\widetilde{X}$ is of Fano type over $A$.

Moreover, if $\phi$ is a surjective endomorphism of $X$ such that $(X,\Delta)$ is a $\phi$-pair, then we obtain the following commutative diagram
\[
\xymatrix{
X  \ar[d]_{\phi^m}& \widetilde{X} \ar[l]_\mu \ar[r]^{\pi} \ar[d]_{\widetilde{\phi}} & A \ar[d]_{\phi_A} \\
X & \widetilde{X} \ar[l]^\mu \ar[r]_{\pi} & A \\
}
\]
for some $m$, where $\widetilde{\phi}$ and $\phi_A$ are endomorphisms.
\end{thm}

\begin{proof}
Applying Theorem \ref{main thm for pairs} for $(X,0)$, we obtain Theorem \ref{main thm'}.
We note that $(X,0)$ is a pair with respect to all surjective endomorphisms of $X$ and quasi-\'etale cover of $(X,0)$ is \'etale in codimension one.
\end{proof}

The following corollary gives a characterization of Fano type admitting an int-amplified endomorphism.

\begin{cor}[Corolarry \ref{main cor}]\label{main cor'}
Let $X$ be a normal $\Q$-factorial klt projective variety admitting an int-amplified endomorphism.
The following conditions are equivalent to each other.
\begin{enumerate}
    \item $X$ is of Fano type.
    \item $\pi_{1}^{\mathrm{\acute{e}t}}(X_{\mathrm{sm}})$ is finite, where $\pi_{1}^{\mathrm{\acute{e}t}}(X_{\mathrm{sm}})$ is the \'etale fundamental group of the smooth locus of $X$.  
\end{enumerate}
Furthermore, if we further assume that $X$ is smooth, then the following condition is also equivalent.

(3) $X$ is rationally connected.
\end{cor}

\begin{proof}
The condition (1) implies the condition (2) by \cite[Theorem 1.13]{greb-kebekus-peternell}.
We assume that the condition (2).
By Theorem \ref{main thm}, we obtain the following diagram
\[
\xymatrix{
X & \widetilde{X} \ar[l]_\mu \ar[r]^{\pi}  & A \\
}
\]
as in Theorem \ref{main thm}.
Since $\mu$ is quasi-\'etale, $\mu^{-1}(X_{\mathrm{sm}}) \to X_{\mathrm{sm}}$ is \'etale by the Zariski-Nagata purity theorem \cite[Theorem 41.1]{nagata}.
In particular, $\pi^{\mathrm{\acute{e}t}}_1(\widetilde{X}_{\mathrm{sm}})$ is also finite.
Hence $A$ is a point, in particular, $\widetilde{X}$ is of Fano type.
By Proposition \ref{quasi-etale preserves fano ytpe}, $X$ is also of Fano type.

If we further assume that $X$ is smooth rationally connected, then $\pi_{1}^{\mathrm{\acute{e}t}}(X)$ is trivial, so $X$ is of Fano type.
On the other hand, varieties of Fano type are rationally connected by \cite
{hacon-mackernan07}.
\end{proof}

Corollary \ref{main cor'} is an affirmative answer for the following conjecture in the case where $\pi_{1}^{\mathrm{\acute{e}t}}(X_{\mathrm{sm}})$ is finite and $X$ is $\Q$-factorial klt.

\begin{conj}[\textup{\cite[Conjecture 1.2]{broustet-gongyo17}}]
Let $X$ be a normal variety admitting a non-invertible polarized endomorphism.
Then $X$ is of Calabi-Yau type.
\end{conj}

\bibliography{bibliography}
\bibliographystyle{amsalpha}
\end{document}